\documentclass[11pt]{amsart}
\usepackage{geometry}
\usepackage{graphicx}
\usepackage{amssymb}
\usepackage[lite]{amsrefs}
\usepackage{epstopdf}
\usepackage{tikz}
\usepackage[colorlinks]{hyperref}
\usepackage[all]{xy}
\usepackage{mathpazo}
\newtheorem{thm}{Theorem}[section]
\newtheorem{cor}[thm]{Corollary}

\newtheorem{prop}[thm]{Proposition}
\theoremstyle{definition}
\newtheorem{defn}[thm]{Definition}
\theoremstyle{remark}
\newtheorem{rem}[thm]{Remark}
\numberwithin{equation}{section}
\newtheorem{ex}{Example}[section]
\newtheorem*{pr}{Proof}
\newcommand{\set}[1]{\left\{#1\right\}}

\newcommand{\To}{\longrightarrow}

\begin{document}

\title[A categorification for the signed chromatic polynomial]{A categorification for the signed chromatic polynomial}
\author{Zhiyun Cheng}
\address{Laboratory of Mathematics and Complex Systems, School of Mathematical Sciences, Beijing Normal University, Beijing 100875, People's Republic of China}
\email{czy@bnu.edu.cn}

\author{Ziyi Lei}
\address{School of Mathematical Sciences, Beijing Normal University, Beijing 100875, People's Republic of China}
\email{201711130219@mail.bnu.edu.cn}

\author{Yitian Wang}
\address{School of Mathematical Sciences, Beijing Normal University, Beijing 100875, People's Republic of China}
\email{201711130205@mail.bnu.edu.cn}

\author{Yanguo Zhang}
\address{School of Mathematical Sciences, Beijing Normal University, Beijing 100875, People's Republic of China}
\email{201711130208@mail.bnu.edu.cn}

\begin{abstract}
By coloring a signed graph by signed colors, one obtains the signed chromatic polynomial of the signed graph. For each signed graph we construct graded cohomology groups whose graded Euler characteristic yields the signed chromatic polynomial of the signed graph. We show that the cohomology groups satisfy a long exact sequence which corresponds to signed deletion-contraction rule. This work is motivated by Helme-Guizon and Rong's construction of the categorification for the chromatic polynomial of unsigned graphs.
\end{abstract}
\maketitle
\noindent{\textbf{keywords}\quad signed graph, signed chromatic polynomial, chromatic cohomology}

\section{Introduction}
The chromatic polynomial, which encodes the number of distinct ways to color the vertices of a graph, was introduced by Birkhoff in attempt to attack the four-color problem \cite{Birkhoff, BL}. Birkhoff's definition is limited to the planar graphs, later Whitney extended this notion to nonplanar graphs. Although this polynomial did not lead to a solution to the four-color problem, it is one of the most important polynomials in graph theory. The reader is referred to \cite{RCR} for a nice introduction to the chromatic polynomial and \cite{JH} for some recent breakthrough.

In recent years, there are a lot of works on chromatic polynomial and its categorification. Motivated by Khovanov's seminal work on the categorification of Jones polynomial \cite{MK}, Helme-Guizon and Rong introduced a categorification for the chromatic polynomial by constructing graded cohomology groups whose graded Euler characteristic is equal to the chromatic polynomial of the graph \cite{LY}. Later, Jasso-Hernandez and Rong introduced a categorification for the Tutte polynomial \cite{HR}. By using the similar idea, Luse and Rong proposed a categorification for the Penrose polynomial, and put forward some relations with other categorifications \cite{KY}. Recently, Sazdanovic and Yip constructed a categorification of the chromatic symmetric function \cite{SY}, which can be considered as a generalization of the chromatic polynomial.

The chromatic cohomology was well studied during the past fifteen years. For example, in \cite{CCR} M. Chmutov, S. Chmutov and Y. Rong proved the knight move theorem for chromatic cohomology. It follows that the ranks of the cohomology groups are completely determined by the chromatic polynomial. We remark that the original knight move conjecture is false for Khovanov homology \cite{MM}. The reader is referred to \cite{GPR, LS} for some investigations on the torsion in chromatic cohomology.

A signed graph is a graph in which each edge is labeled with a positive sign or a negative sign. The signed graph coloring was first studied by Cartwright and Harary in \cite{CH}. In early 1980's, Zaslavsky tried to use signed colors to color signed graphs \cite{T}.  The main principle of how to color a signed graph is equivalent signed graphs have the same chromatic number. Here two signed graphs are said to be equivalent if they are related by finitely many vertex switchings. Zalavsky found some properties of signed graphs and introduced two kinds of chromatic polynomial, say the chromatic polynomial and the balanced chromatic polynomial. Recently, a good survey on this topic was written by Steffen and Vogel \cites{EA}.

It's natural to ask whether we can define a categorification for the chromatic polynomial and the balanced chromatic polynomial of signed graphs, following the categorification for the chromatic polynomial of unsigned graphs. The main aim of this paper is to construct such two categorifications. Furthermore, we can also put forward the so-called signed deletion-contraction rule. We show that the cohomology groups satisfy a long exact sequence corresponding to it, which is based on the corresponding exact sequence in Helme-Guizon and Rong's work \cite{LY}.

The rest of this paper is arranged as follows. Section \ref{section2} is devoted to give a brief introduction to the chromatic polynomial. In Section \ref{section3}, we give a quick review of basics of signed graphs and signed graph colorings. Then we recall the notion of signed chromatic polynomial, which combines the chromatic polynomial and the balanced chromatic polynomial of signed graphs. In the beginning of Section \ref{section4}, we recall the notion of graded dimension of graded $\mathbb{Z}$-modules, then construct the categorifications for the chromatic polynomial and the balanced chromatic polynomial. Several examples are also given. Some basic properties of the cohomology groups in these two categorifications are discussed in Section \ref{section5}.

\section{The chromatic polynomial}\label{section2}
 We begin our discussion with a brief introduction to the chromatic polynomial. We shall consistently use $G$ to denote a graph, and use $V(G), E(G)$ to denote its vertex set and edge set respectively. A \emph{proper coloring} of $G$ is an assignment of elements from a color set to $V(G)$, such that adjacent vertices receive different colors. In other words, a proper coloring of $G$ is a map $\varphi$ from $V(G)$ to a color set $C$, which requires that for any $v_1, v_2 \in V(G)$, if there exists an edge $e\in E(G)$ connecting $v_1$ and $v_2$, then $\varphi(v_1)\neq\varphi(v_2)$. If the color set $C=\{1, 2, \cdots, \lambda\}$, then we denote the number of all proper colorings of $G$ by $P_G(\lambda)$. It follows immediately that if $G$ contains a loop, i.e. an edge that connects a vertex to itself, then $P_G(\lambda)=0$. By using the deletion-contraction relation
$$P_{G}(\lambda)=P_{G-e}(\lambda)-P_{G/e}(\lambda),$$
it is not difficult to find that actually $P_G(\lambda)$ defines a polynomial \cite{RCR}, which is called the \emph{chromatic polynomial} of $G$. Here $G-e$ denotes the graph obtained from $G$ by removing the edge $e$, and $G/e$ is the graph obtained from $G$ by contracting the edge $e$.

It's obvious that $P_{N_k}=\lambda^{k}$, if $N_k$ is the graph consists of $k$ vertices but zero edges. Together with the deletion-contraction relation, these two relations uniquely determines $P_G(\lambda)$. As an example, the chromatic polynomial of $P_3=$
 \begin{tikzpicture}[baseline=0]
  \draw[fill] (0,0.5) circle (.07);
  \draw[fill] (0.5,-0.3) circle (.07);
  \draw[fill] (-0.5,-0.3) circle (.07);
  \draw (0,0.5)--(0.5,-0.3)--(-0.5,-0.3)--(0,0.5);
 \end{tikzpicture} can be calculated as follows
 \begin{equation*}
 	\begin{aligned}
 		P_\lambda\Bigg(\begin{tikzpicture}[baseline=0]
  \draw[fill] (0,0.5) circle (.07);
  \draw[fill] (0.5,-0.3) circle (.07);
  \draw[fill] (-0.5,-0.3) circle (.07);
  \draw (0,0.5)--(0.5,-0.3)--(-0.5,-0.3)--(0,0.5);
 \end{tikzpicture}\Bigg)
 &=
 P_\lambda\Bigg(\begin{tikzpicture}[baseline=0]
  \draw[fill] (0,0.5) circle (.07);
  \draw[fill] (0.5,-0.3) circle (.07);
  \draw[fill] (-0.5,-0.3) circle (.07);
  \draw (0,0.5)--(0.5,-0.3)(-0.5,-0.3)--(0,0.5);
 \end{tikzpicture}\Bigg)
 -
 P_\lambda\Bigg(\begin{tikzpicture}[baseline=0]
  \draw[fill] (0,0.5) circle (.07);
  \draw[fill] (0,-0.3) circle (.07);
  \draw (0,0.5)..controls (0.3,0.1)..(0,-0.3);
  \draw (0,-0.3)..controls (-0.3,0.1)..(0,0.5);
 \end{tikzpicture}\Bigg)\\
 &=
 P_\lambda\Bigg(\begin{tikzpicture}[baseline=0]
  \draw[fill] (0,0.5) circle (.07);
  \draw[fill] (0.5,-0.3) circle (.07);
  \draw[fill] (-0.5,-0.3) circle (.07);
  \draw (0,0.5)--(0.5,-0.3)(-0.5,-0.3)--(0,0.5);
 \end{tikzpicture}\Bigg)
 -
 P_\lambda\Bigg(\begin{tikzpicture}[baseline=0]
  \draw[fill] (0,0.5) circle (.07);
  \draw[fill] (0,-0.3) circle (.07);
  \draw (0,0.5)--(0,-0.3);
 \end{tikzpicture}\Bigg)\\
 &=
 P_\lambda\Bigg(\begin{tikzpicture}[baseline=0]
  \draw[fill] (0,0.5) circle (.07);
  \draw[fill] (0.5,-0.3) circle (.07);
  \draw[fill] (-0.5,-0.3) circle (.07);
  \draw (0,0.5)--(0.5,-0.3)(-0.5,-0.3)(0,0.5);
 \end{tikzpicture}\Bigg)
 -
 2P_\lambda\Bigg(\begin{tikzpicture}[baseline=0]
  \draw[fill] (0,0.5) circle (.07);
  \draw[fill] (0,-0.3) circle (.07);
  \draw (0,0.5)--(0,-0.3);
 \end{tikzpicture}\Bigg)\\
 &=
 (P_\lambda\Bigg(\begin{tikzpicture}[baseline=0]
  \draw[fill] (0,0.5) circle (.07);
  \draw[fill] (0.5,-0.3) circle (.07);
  \draw[fill] (-0.5,-0.3) circle (.07);
 \end{tikzpicture}\Bigg)
 -
 P_\lambda\Bigg(\begin{tikzpicture}[baseline=0]
  \draw[fill] (0,0.5) circle (.07);
  \draw[fill] (0,-0.3) circle (.07);
 \end{tikzpicture}\Bigg))
 -2(
 P_\lambda\Bigg(\begin{tikzpicture}[baseline=0]
  \draw[fill] (0,0.5) circle (.07);
  \draw[fill] (0,-0.3) circle (.07);
 \end{tikzpicture}\Bigg)
 -
 P_\lambda(\begin{tikzpicture}[baseline=0]
  \draw[fill] (0,0.1) circle (.07);
 \end{tikzpicture}))\\
 &=
 P_\lambda\Bigg(\begin{tikzpicture}[baseline=0]
  \draw[fill] (0,0.5) circle (.07);
  \draw[fill] (0.5,-0.3) circle (.07);
  \draw[fill] (-0.5,-0.3) circle (.07);
 \end{tikzpicture}\Bigg)
 -
 3P_\lambda\Bigg(\begin{tikzpicture}[baseline=0]
  \draw[fill] (0,0.5) circle (.07);
  \draw[fill] (0,-0.3) circle (.07);
 \end{tikzpicture}\Bigg)
 +2
 P_\lambda(\begin{tikzpicture}[baseline=0]
  \draw[fill] (0,0.1) circle (.07);
 \end{tikzpicture})
 \\
 &=\lambda^3-3\lambda^2+2\lambda
 	\end{aligned}
 \end{equation*}
 
On the other hand, according to the principle of inclusion-exclusion, there is another formula for $P_G(\lambda)$. For each $s\subseteq E(G)$, denote $[G: s]$ the subgraph whose vertex is $V(G)$ and edge set is $s$, and let $k(s)$ be the number of connected components of $[G: s]$. Then we have
$$P_{G}(\lambda)=\sum_{s \subseteq E(G)}(-1)^{|s|} \lambda^{k(s)}=\sum_{i \geq 0}(-1)^{i} \sum_{s \subseteq E(G),|s|=i} \lambda^{k(s)}.$$
Note that $\lambda^{k(s)}$ is nothing but the polynomial which counts the ways of colorings of $[G: s]$ such that adjacent vertices have the same color. The formula above plays an important role in the categorification of the chromatic polynomial, see \cite{LY} for more details.
 
\section{Signed graph and signed chromatic polynomials}\label{section3}
\subsection{Signed graph and signed colorings}
First, let's take a quick review of signed graphs. Let $SG=(G,\sigma)$ be a signed graph on an ordinary graph $G=(V(G),E(G))$ with a sign on each edge $\sigma:\,E(G)\to \set{1,-1}$. For each subgraph of $SG$, we call it \emph{unbalanced} if there is a negative circuit, on which the product of edges' signs is negative. Otherwise we say it is \emph{balanced}. Let us use $b(SG)$ to denote the number of balanced components of $SG$. The following is an example.
\begin{ex}
For the signed graph in Figure \ref{fig:1}, there are only 2 negative circuits, say, $v_1v_3v_4v_1$ and $v_1v_3v_4v_2v_1$. Both of them come from the left component. If follows that the left component is unbalanced and the right one is balanced, hence we have $b(SG)=1$.
\end{ex}
 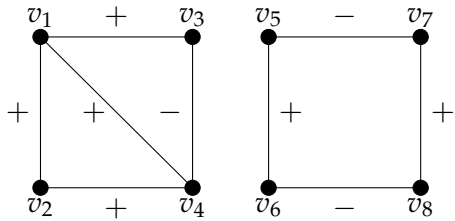
\begin{figure}[htb]
	\centering
	\begin{tikzpicture}
		\draw[fill] (0,0) circle (.1) node [above] {$v_1$};
		\draw[fill] (0,-2) circle (.1) node [below] {$v_2$};
		\draw[fill] (2,0) circle (.1) node [above] {$v_3$};
		\draw[fill] (2,-2) circle (.1) node [below] {$v_4$};

		\draw[fill] (3,0) circle (.1) node [above] {$v_5$};
		\draw[fill] (3,-2) circle (.1) node [below] {$v_6$};
		\draw[fill] (5,0) circle (.1) node [above] {$v_7$};
		\draw[fill] (5,-2) circle (.1) node [below] {$v_8$};

		\draw (0,0) -- node [left] {$+$} (0,-2);
		\draw (0,0) -- node [above] {$+$} (2,0);
		\draw (0,0) -- node [left] {$+$} (2,-2);
		\draw (0,-2) -- node [below] {$+$} (2,-2);
		\draw (2,0) -- node [left] {$-$} (2,-2);

		\draw (3,0) -- node [right] {$+$} (3,-2);
		\draw (3,0) -- node [above] {$-$} (5,0);
		\draw (3,-2) -- node [below] {$-$} (5,-2);
		\draw (5,0) -- node [right] {$+$} (5,-2);
	\end{tikzpicture}
	\caption{A signed graph with one balanced component and one unbalanced component}
	\label{fig:1}
\end{figure}

In \cite{T}, Zalavsky introduced the signed coloring on signed graph. Recently, this notion was modified by E. M\'{a}\v{c}ajov\'{a}, A. Raspaud and M. \v{S}koviera \cite{MRS} to make the corresponding chromatic number coincides with the classical chromatic number when all the signs of the signed graph are positive. A signed $\lambda$-coloring is map $k:V(SG)\to K$, where $K=\set{-\mu,\cdots,-1,0,1,\cdots,\mu}$ if $\lambda=2\mu+1$, and $K=\set{-\mu,\cdots,-1,1,\cdots,\mu}$ (which is called \emph{zero-free}), if $\lambda=2\mu$. We call $k$ a \emph{proper $\lambda$-coloring} (or just coloring for short) if for each edge $e=v_1v_2$ (it is possible that $v_1=v_2$), we have $k(v_1)\neq\sigma(e)k(v_2)$. The \emph{signed chromatic number} is defined to be the smallest $\lambda$ such that $SG$ admits a proper $\lambda$-coloring.

Let $SG=(G,\sigma)$ be a signed graph, each function $u:\,V(SG)\to\set{-1,1}$ leads to a switching function on $SG$: for any $e\in E(SG)$ connecting $v_1,v_2\in V(SG)$, $\sigma(e)$ will be changed into $u(v_1)\sigma(e)u(v_2)$. In particular, if the edge $e$ is a loop, then the sign $\sigma(e)$ is preserved. If a switching function sends a vertex $v$ to $-1$ and all other vertices to $1$, we call this operation a \emph{vertex switching} on $v$ and denote the new signed graph obtained by $v(SG)$. Two signed graphs are called \emph{equivalent} if they are related by finitely many vertex switchings. Note that vertex switching preserves the sign of any circuit, hence the number of balanced components is invariant under vertex switching. On the other hand, it is obvious that if we are given a proper coloring of $SG$, after applying the vertex switching on some vertex $v\in V(SG)$, the coloring obtained from the original coloring by reversing the sign of $k(v)$ is a proper coloring for $v(SG)$. It follows that for any given $\lambda$, equivalent signed graphs have the same number of proper colorings. In particular, if we are given a balanced signed graph, since a balanced signed graph is equivalent to the corresponding positive graph, it suffices to consider the proper colorings on the unsigned graph.

\subsection{Signed chromatic polynomials}
Similar to the ordinary graph, the number of proper signed colorings with $\lambda$ (odd or even) colors is a polynomial respect to $\lambda$, we call it the \emph{signed chromatic polynomial} and use $P_{SG}(\lambda)$ to denote it. In particular, if $\lambda=2\mu +1\,(\mu\in\mathbb{Z})$, we call it the \emph{chromatic polynomial}, otherwise we call it the \emph{balanced chromatic polynomial}, as suggested in \cite{T}.

As we mentioned before, a switching function does not change the number of proper colorings, hence it also preserves the signed chromatic polynomial. Similar to the classical chromatic polynomial, we have the following deletion-contraction relation.

\begin{prop}\label{proposition3.1}
Let $SG$ be a signed graph and $e$ a positive edge. We use $SG-e$ to denote the signed graph obtained from $SG$ by deleting $e$, and use $SG/e$ to denote the signed graph obtained from $SG$ by contracting $e$. The deletion-contraction relation of the signed graph $SG$ with respect to $e$ reads
$$P_{SG}(\lambda)=P_{SG-e}(\lambda)-P_{SG/e}(\lambda).$$
\end{prop}

Notice that the positive edge $e$ above could be a loop, in which case $P_{SG}(\lambda)=0$. If $e$ is a negative edge but not a loop, we can take vertex switching on one of its endpoint first and then apply the deletion-contraction relation. 

Let us use $SN_m^n$ $(m\geq n)$ to denote the graph with $m$ vertices and $n$ negative loops, where each negative loop joins one vertex to itself. The signed chromatic polynomial of $SN_m^n$ can be calculated directly
\begin{equation*}
 P_{SN_m^n}(\lambda)=
\begin{cases}
 	{\lambda^{m-n}(\lambda-1)^{n}}, & {\lambda\text{ is odd; }}\\
 	{\lambda^m}, & {\lambda\text{ is even.}}\\
\end{cases}
\end{equation*}

By using the deletion-contraction relation and the signed chromatic polynomial of $SN_m^n$, we can calculate the signed chromatic polynomial of any signed graph, and verify that both the chromatic polynomial and the balanced chromatic polynomial are well defined as polynomials. We need to point out that the signed chromatic polynomial, as a joint name for both, is not always a polynomial strictly. As an example, let us calculate the signed chromatic polynomial of $SP_3=
\begin{tikzpicture}[baseline=0]
	\draw[fill] (0,0.5) circle (.07);
  	\draw[fill] (0.5,-0.3) circle (.07);
  	\draw[fill] (-0.5,-0.3) circle (.07);
  	\draw (0,0.5)--node[right]{$+$}(0.5,-0.3);
  	\draw (0,0.5)--node[left]{$+$}(-0.5,-0.3);
  	\draw (-0.5,-0.3)--node[below]{$-$}(0.5,-0.3);
\end{tikzpicture}$.
\begin{ex}
 	\begin{equation*}
 		\begin{aligned}
 			P_\lambda(SP_3)&=
P_\lambda\Bigg(\begin{tikzpicture}[baseline=0]
	\draw[fill] (0,0.5) circle (.07);
  	\draw[fill] (0.5,-0.3) circle (.07);
  	\draw[fill] (-0.5,-0.3) circle (.07);
  	\draw (0,0.5)--node[right]{$+$}(0.5,-0.3);
  	\draw (0,0.5)--node[left]{$+$}(-0.5,-0.3);
\end{tikzpicture}\Bigg)
-
P_\lambda\Bigg(\begin{tikzpicture}[baseline=0]
	\draw[fill] (0,0.5) circle (.07);
  	\draw[fill] (0,-0.3) circle (.07);
  	\draw (0,0.5)..controls(0.2,0.1)..node[right]{$-$}(0,-0.3);
  	\draw (0,0.5)..controls(-0.2,0.1)..node[left]{$+$}(0,-0.3);
\end{tikzpicture}\Bigg)
\\
&=
P_\lambda\Bigg(\begin{tikzpicture}[baseline=0]
	\draw[fill] (0,0.5) circle (.07);
  	\draw[fill] (0.5,-0.3) circle (.07);
  	\draw[fill] (-0.5,-0.3) circle (.07);
  	\draw (0,0.5)--node[right]{$+$}(0.5,-0.3);
  	\draw (0,0.5)--node[left]{$+$}(-0.5,-0.3);
\end{tikzpicture}\Bigg)
-
P_\lambda\Bigg(\begin{tikzpicture}[baseline=0]
	\draw[fill] (0,0.5) circle (.07);
  	\draw[fill] (0,-0.3) circle (.07);
  	\draw (0,0.5)--node[left]{$+$}(0,-0.3);
\end{tikzpicture}\Bigg)
+
P_\lambda\Bigg(\begin{tikzpicture}[baseline=0]
	\draw[fill] (0,0.1) circle (.07);
	\draw (-0.3,0.1) circle (.3);
	\node at (-0.6,0.2) [left] {$-$};
\end{tikzpicture}\Bigg)
\\
&=\lambda(\lambda-1)^2-\lambda(\lambda-1)+
P_\lambda\Bigg(\begin{tikzpicture}[baseline=0]
	\draw[fill] (0,0.1) circle (.07);
	\draw (-0.3,0.1) circle (.3);
	\node at (-0.6,0.2) [left] {$-$};
\end{tikzpicture}\Bigg)
\\
&=
\begin{cases}
	\lambda^3-3\lambda^2+3\lambda-1,&\lambda \text{ is odd;}\\
	\lambda^3-3\lambda^2+3\lambda,&\lambda \text{ is even.}\\
\end{cases}
 		\end{aligned}
 	\end{equation*}
\label{ex3.2}
\end{ex}

Similar to the unsigned graph, there is another formula for the signed chromatic polynomial deduced from the principle of inclusion-exclusion.

\begin{prop}\label{prop3.4}
Let $SG$ be a signed graph on $G$, we define $Q_{SG}(\lambda)$ to be the polynomial which calculates the ways to coloring $SG$ such that for any $e\in E(G)$ connecting $v_1,\,v_2\in V(G)$, it satisfies $k(v_1)=\sigma(e)k(v_2).$
Then we have 
\begin{equation*}
Q_{SG}(\lambda)=
\begin{cases}
{\lambda^{b(SG)}}, & {\lambda\text{ is odd or $SG$ is balanced;}} \\
0 , & {\text{else.}}\\
\end{cases}
\end{equation*}
\end{prop}
\begin{pr}
For any coloring satisfies the condition above, the coloring of one component is completely determined by the color on one vertex. If the component is balanced then there are $\lambda$ colors to choose. However, if the component is unbalanced, we can only assign $0$ to it. In other words, we can not color it when $\lambda$ is even and there is only one way to color it when $\lambda$ is odd. Then we obtain
$$Q_{SG}(\lambda)=
\begin{cases}
{\lambda^{b(SG)}}, & {\lambda\text{ is odd ;}} \\
{\lambda^{b(SG)}}\cdot 0^{n-b(SG)}, & {\lambda\text{is even.}}\\
\end{cases}$$
where $n$ is the number of components of $SG$ and as usual we set $0^0=1$. The proof is finished.
\end{pr}

\begin{prop}
According to the principle of inclusion-exclusion, we have
$$P_{SG}(\lambda)=\sum_{i \geq 0}(-1)^{i} \sum_{s\subseteq E(G),|s|=i} Q_{[SG: s]}(\lambda),$$
where $[SG: s]$ is the signed subgraph on $[G: s]$.
\end{prop}

\begin{pr}
We obtain this directly by applying the principle of inclusion-exclusion together with the definition of $Q_{SG}(\lambda)$.
\end{pr}

\section{A categorification for the signed chromatic polynomial}\label{section4}
As we mentioned above, the signed chromatic polynomial itself is not always a polynomial. However, the chromatic polynomial and the balanced chromatic polynomial are both well defined as polynomials. So in this section, we give the categorifications for these two polynomials respectively. We first recall the definition and some properties of the graded dimension (also called quantum dimension) of a graded $\mathbb{Z}$-module.
\subsection{Graded dimension of graded modules}
\begin{defn}
Let $\mathcal{M}=\bigoplus\limits_j M_j$ be a graded $\mathbb{Z}$ -module, where $M_j$ denotes the set of homogeneous elements with degree $j$. The \emph{graded dimension} (also called \emph{quantum dimension}) of $\mathcal{M}$ is the power series
	$$
	q \text{dim} \mathcal{M}:=\sum_{j} q^{j} \cdot \text{rank}\left(M_{j}\right),
	$$
where $\text{rank}\left(M_{j}\right)=\text{dim}_{\mathbb{Q}}\left(M_{j} \otimes \mathbb{Q}\right)$.
\end{defn}

We remark that the torsion part of $\mathcal{M}$ can not be detected by the graded dimension. Let $\mathcal{M}$ and $\mathcal{N}$ be two graded $\mathbb{Z}$-modules. Then $\mathcal{M} \oplus \mathcal{N}$ and $\mathcal{M} \otimes \mathcal{N}$ are both graded $\mathbb{Z}$ -modules and the graded dimensions can be obtained as below
$$q\text{dim}(\mathcal{M} \oplus \mathcal{N})=q \text{dim}(\mathcal{M})+q \text{dim}(\mathcal{N}), q\text{dim}(\mathcal{M} \otimes \mathcal{N})=q \text{dim}(\mathcal{M}) \cdot q \text{dim}(\mathcal{N}).$$

The following example of graded $\mathbb{Z}$-module, which is taken from \cite{LY}, will be frequently used throughout this paper. It was originally used in the construction of Khovanov homology \cite{MK}.

\begin{ex}
Let $M$ be the graded free $\mathbb{Z}$-module with two basis elements $1$ and $x$, whose degrees are $0$ and $1$ respectively. According to the definition of the graded dimension, we have
$$q \text{dim}(M)=q^0\cdot \text{rank}(\mathbb{Z})+q^1\cdot\text{rank}(\mathbb{Z}x)=1+q.$$
On the other hand, by using the left identity above, as $M=\mathbb{Z} \oplus \mathbb{Z} x$, we obtain the same result
$$q \text{dim}(M)=q\text{dim}(\mathbb{Z})+q\text{dim}(\mathbb{Z}x)=1+q.$$
Additionally, the tensor product formula above tells us that $q\text{dim}M^{\otimes k}=(1+q)^k$.
\end{ex}
\begin{defn}
Let $\set{\ell}$ be the ``degree shift'' operation on graded $\mathbb{Z}$-modules. That is, if $\mathcal{M}=\oplus_{j} M_{j}$ is a graded $\mathbb{Z}$ -module where $M_{j}$ denotes the set of elements of $\mathcal{M}$ of degree $j,$ we set $\mathcal{M}\{\ell\}_{j}:=M_{j-\ell}$ so that $q \text{dim}\mathcal{M}\{\ell\}=$$q^{\ell} \cdot q \text{dim} \mathcal{M}$. In other words, all the degrees are increased by $\ell$.
\end{defn}
For example, since $\text{deg}x=1$, then $\mathbb{Z}\set{1}=\mathbb{Z}x$. And for each $\ell\in\mathbb{N}$, it's easy to check that $\mathcal{M} \otimes \mathbb{Z}\{\ell\} \cong \mathcal{M}\{\ell\}$, the $\mathbb{Z}$-module isomorphic to $\mathcal{M}$ with degree of every homogeneous element raised up by $\ell$.

\subsection{A categorification for the chromatic polynomial}
In this section, we give a categorification for the chromatic polynomial, i.e. the signed chromatic polynomial with odd $\lambda$. In order to do this, we need to introduce a sequence of graded $\mathbb{Z}$-modules and graded differentials.
\subsubsection{Cochain groups on signed graphs}
Let $SG$ be a signed graph on $G$, for each signed spanning graph $[SG: s]$ led by $s\subseteq E(SG)$, we assign a graded $\mathbb{Z}$-module $M_s(SG)$ as follows: we assign a copy of $M$ to each balanced component and a copy of $\mathbb{Z}$ to the unbalanced component and then take the tensor product. In this way, it is guaranteed that $q\text{dim} M_{s}(SG)=Q_{[G: s]}(1+q)$ if $q=\lambda-1\in2\mathbb{N}$, which further ensures the following result.
\begin{prop}\label{Theorem4.4}
Let $SG$ be a signed graph. For each signed spanning graph $[SG: s]$ led by $s\subseteq E(SG)$, we define the cochain group $C^i(SG)=\mathop\bigoplus\limits_{s\subseteq E(SG),|s|=i}M_s(SG)$, then $P_{SG}(1+q)=\sum\limits_{i\geq 0}(-1)^i q\text{dim} C^i(SG)$, provided that $q\in2\mathbb{N}$.
\end{prop}

Now we have a cochain group $C^{\bullet}(SG)$ whose graded Euler characteristic is equal to the chromatic polynomial with $\lambda=1+q$, where $\lambda$ is odd and $q$ is even. The next step is to introduce a differential $d_s$ for the chain complex which satisfies $d_s^2=0$.
  
\subsubsection{The differential}
We first recall the definition of enhanced state of an unsigned graph \cite{LY}. Let $G=(V(G),E(G))$ be a graph with an ordering on $E(G)$. An \emph{enhanced state} $S=\{s, c\}$ consists of a subset $s\subset E(G)$ and an assignment $c$ which assigns $1$ or $x$ to each component of $[G: s]$. For each enhanced state $S$, we set $i(S)=|s|$ and $j(S)$ to be the number of $x$ in $c$. 

We define a multiplication $m: M\otimes M\to M$ by $m(1\otimes1)=1$, $m(1\otimes x)=m(x\otimes 1)=x$ and $m(x\otimes x)=0$. We remark that actually there is a Frobenius algebra structure on $M$ \cite{MK}. However, here the only operation we need is the multiplication, since adding an edge never increases the component number.

We set $C^{i, j}(G)=\text{span}\langle S|S\text{ is an enhanced state of }G \text{ with } i(S)=i, j(S)=j\rangle$, where the span is taken over $\mathbb{Z}$. The differential $d: C^{i, j}(G)\to C^{i+1,j}(G)$ is defined as $$d: S=(s, c)\to\sum_{e\in E(G)-s}(-1)^{n(e)}S_e,$$ where $n(e)$ is the number of edges in $s$ that are ordered before $e$. Here $S_e$ denotes either an enhanced state or 0, which is defined as follows
\begin{itemize}
\item if $e$ connects a component $E_i$ to itself, then the components of $[G: s\cup\{e\}]$ are $E_{1}, \cdots, E_{i}\cup\{e\}, \cdots, E_{k(s)}.$ We define $c_{e}(E_1)=c(E_1), \cdots, c_{e}(E_{i}\cup\{e\})=c(E_{i}), \cdots, c_{e}(E_{k(s)})=c(E_{k(s)})$.
\item If $e$ connects two components $E_i$ and $E_j$ $(i<j)$, then the components of $[G: s\cup\{e\}]$ are $E_1, \cdots, E_{i-1}, E_i\cup E_j\cup\{e\}, E_{i+1}, \cdots, E_{j-1}, E_{j+1}, \cdots, E_{k(s)}$. We define $c_e(E_i\cup E_j \cup\{e\})=m(c(E_i)\otimes c(E_j))$ and $c_{e}(E_l)=c(E_l)$ if $l\neq i, j$.
\end{itemize}
In particular, if $c(E_i)=c(E_j)=x$, then $c_{e}(E_i\cup E_j\cup\{e\})=m(x\otimes x)=0$. In this case, we set $S_e=0$. Otherwise, $c_e$ is a coloring and we define $S_e$ to be the enhanced state $(s_e=s\cup\{e\}, c_e)$. It was proven in \cite{LY} that $d^2=0$. Therefore $C^{\bullet}(G)=\{\bigoplus\limits_{j\geq 0}C^{i, j}(G), d\}$ is a graded cochain complex with graded Euler characteristic the chromatic polynomial $P_G(\lambda)$ evaluated at $\lambda=1+q$. The corresponding cohomology groups $H^{\bullet}(G)$ are called the \emph{chromatic cohomology groups} of $G$.

Now we turn to the enhanced states for the chromatic polynomial of signed graphs. As before, for a given signed graph $SG$, we use $G$ to denote the corresponding unsigned graph.
\begin{defn}
An enhanced state $S=(s, c)$ of $G$ is an \emph{enhanced state} of $SG$ with respect to chromatic polynomial if and only if $c$ assigns $1$ for all unbalanced components of $[G: s]$. Now we can define $C^{i, j}(SG)$ as follows
\begin{center}
$C^{i, j}(SG)=\text{span}\langle S|S\text{ is an enhanced state of $SG$ with }i(S)=i, j(S)=j \rangle,$
\end{center}
where $i(S)=|s|$ and $j(S)=$ the number of components that  $x$ is assigned to.
\end{defn}
It follows immediately that $C^{i}(SG)=\bigoplus\limits_{j\geq 0}C^{i, j}(SG).$ In order to define the differential $d_s$, we need to define a map $f$ from $C^{i, j}(G)$ to $ C^{i, j}(SG)$.
\begin{defn}\label{definition4.5}
Let $SG$ be a signed graph on $G$, for each enhanced state $S=(s, c)$ in $C^{i, j}(G)$, we define
\begin{equation*}
    	f(S)=
		\begin{cases}
			S, &{\text{if } c \text{ assigns 1 to all unbalanced components}}\\
			0, &{\text{otherwise}}\\
		\end{cases}
  	\end{equation*}
\end{defn}
It's obvious that $f$ extends to a linear map from $C^{i, j}(G)$ onto $ C^{i, j}(SG)$, as well as $C^i(G)$ onto $ C^i(SG)$. We will use $f$ to denote both of them, if there is no confusion. By using $f$, we can define the differential $d_s$ according to the following commutative diagram.
\begin{equation*}
  	\xymatrix{
	C^n(G) \ar[rr]^{d}\ar[d]_{f} &   & C^{n+1}(G) \ar[d]_{f} \\
       C^n(SG)\ar[rr]^{d_s}           &   & C^{n+1}(SG)}
\end{equation*}
\begin{prop}\label{prop4.10}
The differential $d_s: C^n(SG)\to C^{n+1}(SG),\text{ defined by } d_s=f\circ d\circ f^{-1}$ is well defined.
\end{prop}
\begin{proof}
Because both $f, d$ are linear functions, it suffices to prove that for each $z\in C^n(G)$ with $f(z)=0$, we have $f\circ d(z)=0$. Recall that $C^n(G)=\mathop\bigoplus\limits_{s\subseteq E(G),|s|=n}M_s(G)$, where $M_s(G)=M^{\otimes k(s)}$, we suppose $z=\sum\limits_{i=1}^ma_iS_i$ $(a_i\neq0)$, where $S_i=(s_i, c_i)$ $(1\leq i\leq m)$ are distinct enhanced states. Since $f(z)=0$, it follows that for each $1\leq i\leq m$, there exists at least one unbalanced component of $[G: s_i]$ which is assigned an $x$. For $S_1$, let us choose an unbalanced component $E_1$ with $c_1(E_1)=x$. We break down into three cases.
\begin{enumerate}
\item If an edge $e\in E(G)-s_1$ is not adjacent to $E_1$, the the corresponding assignment of $(S_1)_e$ also assigns an $x$ to $E_1$. Hence $(S_1)_e$ is also sent to 0 according to the definition of $f$.
\item If an edge $e\in E(G)-s_1$ connects $E_1$ with another component $E_2$, then the new component $E_1\cup E_2\cup \{e\}$ in $[G: s_1\cup\{e\}]$ is unbalanced. If $c_1(E_2)=1$, then $(c_1)_e(E_1\cup E_2\cup \{e\})=m(x\otimes 1)=x$, hence $f((S_1)_e)=0$. If $c_1(E_2)=x$, then $(c_1)_e(E_1\cup E_2\cup \{e\})=m(x\otimes x)=0$, which is also mapped to 0 under $f$.
\item If an edge $e\in E(G)-s_1$ joins $E_1$ to itself, then the new component $E_1\cup \{e\}$ in $[G: s_1\cup\{e\}]$ is also unbalanced and $(c_1)_e(E_1\cup \{e\})=c_1(E_1)=x$, it follows that $f((S_1)_e)=0$.
\end{enumerate}
For other $S_i$, the proof is analogous to that of $S_1$. In summary, we have $f\circ d(z)=0 ,$ which means that $d_s=f\circ d\circ f^{-1} $ is well defined.
\end{proof}

As $d^2=0$ and $d_s=f\circ d\circ f^{-1}$, we can directly obtain $d_s^2=f\circ d^2\circ f^{-1}=0$, then we obtain the following chain complex
\begin{center}
$C^0(SG)\stackrel{d_s}\To C^1(SG)\stackrel{d_s}\To C^2(SG)\stackrel{d_s}\To\cdots\stackrel{d_s}\To C^i(SG)\stackrel{d_s}\To\cdots$.
\end{center}
\begin{defn}
For a given signed graph $SG$, we call the cohomology groups of the cochain complex above the \emph{chromatic cohomology groups} of $SG$ and use $H^i(SG)$ to denote the $i$-th chromatic cohomology group of $SG$.
\end{defn}

Since $d_s$ is degree preserving, the chromatic cohomology group $H^i(SG)$ can be decomposed as $\bigoplus\limits_{j\geq0}H^{i, j}(SG)$. The chromatic polynomial can be recovered by $P_{SG}(1+q)=\sum\limits_{i\geq 0}(-1)^i q\text{dim}C^i(SG)=\sum\limits_{i\geq 0}(-1)^i q\text{dim}H^i(SG)$. The following proposition tells us that chromatic cohomology groups are well defined signed graph invariants.

\begin{prop}\label{proposition4.8}
The chromatic cohomology groups are independent of the order of the edges.
\end{prop}
\begin{proof}
The proof follows the outline given in \cite{LY}. Suppose $E(SG)=\{e_1, \cdots, e_k, e_{k+1}, \cdots, e_n\}$, it suffices to prove that $H^{i}(SG)\cong H^{i}(SG')$, where $SG'$ is the same signed graph as $SG$ but the edges are reordered as $\{e_1, \cdots, e_{k+1}, e_k, \cdots, e_n\}$.

Since $C^i(SG)=\mathop\bigoplus\limits_{s\subseteq E(SG),|s|=i}M_s(SG)$, where $M_s(SG)=M^{\otimes b([SG: s])}$, it is enough to define an isomorphism $g$ restricted on $M_s(SG)$. Consider the map $g_s$ from $M_s(SG)$ to $M_s(SG')$, which is defined to be
\begin{equation*}
g_s=
\begin{cases}
-id, &\text{if } \{e_k, e_{k+1}\}\subset s\\
id, &{\text{otherwise}}\\
\end{cases}
\end{equation*}
Now we define the map $g: C^i(SG)\to C^i(SG')$ as $g=\bigoplus\limits_{s\subseteq E(SG), |s|=i}g_s$. It is a routine exercise to check that $g$ is a chain map which induces an isomorphism between $H^i(SG)$ and $H^i(SG')$.
\end{proof}

We end this subsection with a concrete example.
\begin{ex}\label{ex4.2}
Let us consider the signed graph $SP_3$, where $E(SP_3)=\set{e_1,e_2,e_3}$, and $\sigma(e_1)=\sigma(e_2)=+1$, $\sigma(e_3)=-1$, as Figure \ref{fig:2} shows.
\begin{figure}[h]
\centering
\begin{tikzpicture}[baseline=0]
\draw[fill] (0,1.5) circle (.07) node[right]{$v_1$};
\draw[fill] (1.5,-0.9) circle (.07) node[right]{$v_2$};
\draw[fill] (-1.5,-0.9) circle (.07) node[left]{$v_3$};
\draw (0,1.5)--node[left]{$e_1,+$}(-1.5,-0.9);
\draw (0,1.5)--node[right]{$e_2,+$}(1.5,-0.9);
\draw (-1.5,-0.9)--node[below]{$e_3,-$}(1.5,-0.9);
\end{tikzpicture}
\caption{A signed graph $SP_3$}
\label{fig:2}
\end{figure}
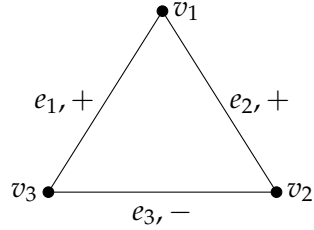

For each enhanced state $S=(s,c)$ of $SP_3$, we denote $s$ and $c$ by elements of $\set{0,1}^3$ and $\set{1,x}^3$, such that the $i$-th position of $s$ is 1 (0) if $e_i\in s$ $(e_i\notin s)$, and all the vertices of the same component have the same color. For example, the enhanced states shown in Figure \ref{fig:3} are $(000, 11x)$, $(100, x1x)$ and $(101, 111)$.
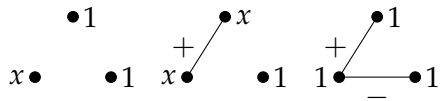
\begin{figure}[htb]
\centering
\begin{tikzpicture}[baseline=0]
\draw[fill] (-2,0.5) circle (.07) node[right]{$1$};
\draw[fill] (-1.5,-0.3) circle (.07) node[right]{$1$};
\draw[fill] (-2.5,-0.3) circle (.07) node[left]{$x$};
\draw (0,0.5)--node[left]{$+$}(-.5,-0.3);
\draw[fill] (0,0.5) circle (.07) node[right]{$x$};
\draw[fill] (0.5,-0.3) circle (.07) node[right]{$1$};
\draw[fill] (-0.5,-0.3) circle (.07) node[left]{$x$};
\draw[fill] (2,0.5) circle (.07) node[right]{$1$};
\draw[fill] (2.5,-0.3) circle (.07) node[right]{$1$};
\draw[fill] (1.5,-0.3) circle (.07) node[left]{$1$};
\draw (2,0.5)--node[left]{$+$}(1.5,-0.3);
\draw (1.5,-0.3)--node[below]{$-$}(2.5,-0.3);
\end{tikzpicture}
\caption{Three enhanced states}
\label{fig:3}
\end{figure}

One computes
\begin{equation*}
\begin{aligned}
        C^0(SP_3)&=&\text{span}\langle&(000,xxx),\\
                &     &                          &(000,xx1),(000,x1x),(000,1xx),\\
                &     &                          &(000,x11),(000,1x1),(000,11x),\\
                &     &                          &(000,111)\rangle;\\
        C^1(SP_3) &=&\text{span}\langle&(001,xxx),(010,xxx),(100,xxx),\\
                &     &                          &(001,x11),(001,1xx),(010,xx1),\\
                &     &                          &(010,11x),(100,x1x),(100,1x1),\\
                &     &                          &(001,111),(010,111),(100,111)\rangle;\\
        C^2(SP_3) &=&\text{span}\langle&(011,xxx),(101,xxx),(110,xxx),\\
                & &                          &(011,111),(101,111),(110,111)\rangle;\\
        C^3(SP_3)&=&\text{span}\langle&(111,111)\rangle.\\
        B^0(SP_3)&=&0;\\
        B^1(SP_3)&=&\text{span}\langle&(001,xxx)+(010,xxx),\\
                &     &                          &(010,xxx)+(100,xxx),\\
                &     &                          &(100,xxx)+(001,xxx),\\
                &     &                          &(100,x1x)+(010,xx1)+(001,x11),\\
                &     &                          &(100,1x1)+(010,xx1)+(001,1xx),\\
                &     &                          &(100,x1x)+(010,11x)+(001,1xx),\\
                &     &                          &(001,111)+(010,111)+(100,111)\rangle;\\
        B^2(SP_3)&=&\text{span}\langle&(110,xxx)+(101,xxx),(101,xxx)+(011,xxx)\\
                &     &                          &(110,111)+(101,111),(101,111)+(011,111)\rangle;\\
        B^3(SP_3)&=&\text{span}\langle&(111,111)\rangle.\\
        Z^0(SP_3)&=&\text{span}\langle&(000,xxx)\rangle;\\
        Z^1(SP_3)&=&\text{span}\langle&(001,xxx),(010,xxx),(100,xxx),\\
                &     &                          &(001,x11)-(001,1xx),\\
                &     &                          &(010,xx1)-(010,11x),\\
                &     &                          &(100,x1x)-(100,1x1),\\
                &     &                          &(001,x11)+(010,11x)+(100,1x1),\\
                &     &                          &(001,111)+(010,111)+(100,111)\rangle;\\
        Z^2(SP_3)&=&\text{span}\langle&(011,xxx),(101,xxx),(110,xxx),\\
                &     &                          &(110,111)+(101,111),(101,111)+(011,111)\rangle;\\
        Z^3(SP_3)&=&\text{span}\langle&(111,111)\rangle.\\
\end{aligned}
\end{equation*}
It follows that
\begin{equation*}
\begin{aligned}
        H^0(SP_3)&=Z^0(SP_3)/B^0(SP_3)\cong\mathbb{Z}\{3\};\\
        H^1(SP_3)&=Z^1(SP_3)/B^1(SP_3)\cong\mathbb{Z}_2\{2\}\mathop\oplus\mathbb{Z}\{1\};\\
        H^2(SP_3)&=Z^2(SP_3)/B^2(SP_3)\cong\mathbb{Z}\{1\};\\
        H^3(SP_3)&=Z^3(SP_3)/B^3(SP_3)\cong0.\\
\end{aligned}
\end{equation*}
By using Proposition~\ref{Theorem4.4}, we can calculate the chromatic polynomial
$$P_{SP_3}(1+q)=\sum_{i\geq0}(-1)^iq\text{dim}C^i(SP_3)=\sum_{i\geq0}(-1)^iq\text{dim}H^i(SP_3)=q^3-q+q=q^3,$$
where $q$ is even. On the other hand, this result can also be obtained by using the deletion-contraction relation on the signed graph, which is shown in Example \ref{ex3.2}.
\end{ex}
  
\subsubsection{A categorification of the deletion-contraction relation}
The aim of this subsection is to introduce a long exact sequence, which recovers the deletion-contraction relation given in Proposition \ref{proposition3.1} if one takes the Euler characteristic of this long exact sequence. So in some sense, the long exact sequence in Corollary \ref{corollary4.10} can be considered as a categorification of the deletion-contraction rule in Proposition \ref{proposition3.1}.

\begin{thm}\label{theorem4.9}
Let $SG$ be a signed graph and $e$ a positive edge, we have the following short exact sequence
\begin{center}
$0\to C^{\bullet-1}(SG/e)\to C^{\bullet}(SG)\to C^{\bullet}(SG-e)\to0$.
\end{center}
\end{thm}

\begin{proof}
Without loss of generality, let us choose an order of $E(SG)$ such that the positive edge $e$ is the first edge. The order of $E(SG)$ induces the an order for $E(SG/e)$ and $E(SG-e)$.

We first introduce a morphism of cochain complexes, say $\widetilde{m}: C^{\bullet}(SG-e)\to C^{\bullet}(SG/e)$. Recall that $C^i(SG)=\mathop\bigoplus\limits_{s\subseteq E(SG), |s|=i}M_s(SG)$, where $M_s(SG)=M^{\otimes b([SG: s])}$, it suffices to consider a fixed subset $s\subseteq E(SG-e)=E(SG/e)$. If $e$ joins a component of $[SG-e: s]$ to itself, then we define $\widetilde{m}$ to be $f_{SG/e}\circ id\circ f_{SG-e}^{-1}$, where $f_{SG/e}$ and $f_{SG-e}$ are similarly defined as $f$ in Definition \ref{definition4.5}. Otherwise, we define $\widetilde{m}$ to be $f_{SG/e}\circ m\circ f_{SG-e}^{-1}$, where $m$ denotes the multiplication $m: M\otimes M\to M$. One can easily check that $\widetilde{m}$ is well defined.

We claim that the complex $C^{\bullet}(SG)$ is the mapping cone of $\widetilde{m}: C^{\bullet}(SG-e)\to C^{\bullet}(SG/e)$. 

Notice that 
\begin{center}
$C^i(SG)=\bigoplus\limits_{s\subseteq E(SG), |s|=i}M_s(SG)=\bigoplus\limits_{e\notin s\subseteq E(SG), |s|=i}M_s(SG)\bigoplus\bigoplus\limits_{e\in s\subseteq E(SG), |s|=i}M_s(SG)$,
\end{center}
and there is an obvious isomorphism between 
\begin{center}
$C^i(SG-e)=\bigoplus\limits_{s\subseteq E(SG-e), |s|=i}M_s(SG-e)$ and $\bigoplus\limits_{e\notin s\subseteq E(SG), |s|=i}M_s(SG)$. 
\end{center}
For the second summand $\bigoplus\limits_{e\in s\subseteq E(SG), |s|=i}M_s(SG)$, one observes that for any $s\in E(SG/e)$, there is a one-to-one correspondence between the components of $[SG/e : s]$ and the components of $[SG: s\cup\{e\}]$. Since the sign of $e$ is positive, then a component of $[SG/e : s]$ is balanced if and only if the corresponding component in $[SG: s\cup\{e\}]$ is balanced. It follows that 
\begin{center}
$\bigoplus\limits_{e\in s\subseteq E(SG), |s|=i}M_s(SG)\cong \bigoplus\limits_{s\subseteq E(SG/e), |s|=i-1}M_s(SG/e)=C^{i-1}(SG/e)$,
\end{center}
and hence
\begin{center}
$C^i(SG)=C^i(SG-e)\bigoplus C^{i-1}(SG/e)$.
\end{center}

Consider the differential $d_s': C^i(SG)\to C^{i+1}(SG)$ given by
$\left(
\begin{array}{ccc}
d_1 &   0   \\
\widetilde{m} & -d_2   \\
\end{array}
\right)$, where $d_1$ denotes the differential on $C^{\bullet}(SG-e)$ and $d_2$ denotes the differential on $C^{\bullet}(SG/e)$. Recall that $e$ is the first edge, it is not difficult to find that $d_s'$ coincides with the differential $d_s$ defined in Proposition \ref{prop4.10}. This finishes the proof of the claim and the result follows directly.
\end{proof}

\begin{cor}\label{corollary4.10}
Let $SG$ be a signed graph and $e$ a positive edge, we have the following long exact sequence
\begin{center}
$\cdots\to H^{i-1}(SG/e)\to H^i(SG)\to H^i(SG-e)\to H^i(SG/e)\to\cdots$.
\end{center}
\end{cor}

\begin{rem}
It would be helpful to describe the two maps $\alpha^*: H^{i-1}(SG/e)\to H^i(SG)$ and $\beta^*: H^i(SG)\to H^i(SG-e)$ intuitively. In order to understand $\alpha^*$, it suffices to consider the map $\alpha: C^{i-1}(SG/e)\to C^i(SG)$. With a given enhanced state $S=(s, c)$ of $SG/e$, notice that $[SG/e: s]$ and $[SG: s\cup\{e\}]$ not only have the same number of components but also the same number of balanced components, since $e$ is positive. The we can define $\alpha(S)=(s\cup\{e\}, c_e)$, which induces the map $\alpha^*: H^{i-1}(SG/e)\to H^i(SG)$. For $\beta^*$, let us choose an enhanced state $S=(s, c)$ of $SG$. If $e\notin s$, then we define $\beta(S)=S$, which is also an enhanced state of $SG-e$, since removing $e$ from a balanced component yields a balanced component. Otherwise, if $e\in s$ then we set $\beta(S)=0$. The map $\beta^*: H^i(SG)\to H^i(SG-e)$ can be induced from $\beta$.
\end{rem}

\subsection{A categorification for the balanced chromatic polynomial}\label{section5}
In this subsection, we discuss how to categorify the balanced chromatic polynomial, i.e. the signed chromatic polynomial with even $\lambda$. 

\subsubsection{Cochain groups on signed graphs}
Let $SG$ be a signed graph on $G$, for each signed subgraph $[SG: s]$ led by $s\subseteq E(SG)$, we assign a graded $\mathbb{Z}$-module $M^b_s(SG)$ as follows:
\begin{itemize}
\item if $[SG: s]$ is balanced, we assigned a copy of $M$ to each component and then take tensor product, i.e. $M^b_s(SG)=M^{\otimes k(s)}$ if $[SG: s]$ is balanced.
\item if $[SG: s]$ is unbalanced, we assigned $0$ to it.
\end{itemize}
In this case, it is guaranteed that $q\text{dim}M^b_s(SG)=Q_{[SG: s]}(1+q)$ when $q$ is odd. As an analogy of Proposition~\ref{Theorem4.4}, we have the following result.
\begin{prop}\label{proposition4.11}
Let $SG$ be a signed graph, we define $C^i_b(SG)=\mathop\bigoplus\limits_{s\subseteq E(G),|s|=i}M_s^b(SG)$, then
$$P_{SG}(1+q)=\sum\limits_{i\geq 0}(-1)^{i}\sum\limits_{s\subseteq E(G),|s|=i}q\text{dim} M^b_{s}(SG)=\sum\limits_{i\geq 0}(-1)^{i}q\text{dim} C_b^i(SG),$$
if $q$ is an odd integer.
\end{prop}
 
\subsubsection{The differential}
\begin{defn}
An enhanced state $S=(s, c)$ of $G$ is an \emph{enhanced state} of $SG$ for balanced chromatic polynomial if and only if $[G: s]$ is balanced.
\end{defn}
As before, we set 
$$C^{i, j}_b(SG)=\text{span}\left\langle S|S\text{ is an enhanced state of } SG\text { with } i(S)=i, j(S)=j\right\rangle,$$ 
where $i(S)=|s|$ and $j(S)$ equals the number of components that $x$ is assigned to. It follows immediately that $C^i_b(SG)=\bigoplus\limits_{j\geq0}C^{i, j}_b(SG)$.

Now we introduce a map $f_b: C^{i, j}(G)\to C_b^{i, j}(SG)$ for balanced chromatic polynomial.
\begin{defn}\label{definition4.13}
Let $SG$ be a signed graph on $G$, for each enhanced state $S=(s, c)$ in $C^{i, j}(G)$, we define
\begin{equation*}
f_b(S)=
\begin{cases}
S,&{[SG: s] \text{ is balanced;}}\\
0,&{\text{otherwise.}}\\
\end{cases}
\end{equation*}
\end{defn}
We extend $f_b$ to a linear projection from $C^{i,j}(G)$ to $ C_b^{i,j}(SG)$, or, from $C^{i}(G)$ to $ C_b^{i}(SG)$ if one sums over $j$. Let us still use $f_b$ to denote it.

By using $f_b$, we define the differential $d_b$ as the below, which is similar to the definition of $d_s$.
\begin{equation*}
\xymatrix{
C^i(G) \ar[rr]^{d}\ar[d]_{f_b} &   & C^{i+1}(G) \ar[d]_{f_b}\\
C^i_b(SG)\ar[rr]^{d_b}         &   & C_b^{i+1}(G)}
\end{equation*}
\begin{prop}
The map $d_b: C_b^i(SG)\to C_b^{i+1}(SG)$ defined by $d_b=f_b\circ d\circ f_b^{-1}$ is well defined.
\end{prop}
\begin{proof}
Similar to the proof of Proposition \ref{prop4.10}, it suffices to show that for any $z=\sum\limits_{i=1}^ma_iS_i$ $(a_i\neq 0)$, if $f_b(z)=0$ then $f_b\circ d(z)=0$. Here $S_i=(s_i, c_i)$ for some $s_i\subset E(SG)$. According to the definition of $f_b$, $[SG: s_i]$ must be unbalanced. Notice that the differential $d$ corresponds to the following two cases
\begin{enumerate}
\item adding an edge connecting an unbalanced component to itself,
\item adding an edge connecting an unbalanced component with another component.
\end{enumerate}
Both of them give rise to a new unbalanced component. It follows that $f_b\circ d(S_i)=0$ and hence $f_b\circ d(z)=0$.
\end{proof}

Since $d^2=0$ and $d_b=f_b\circ d\circ f_b^{-1}$, we obtain $d_b^2=f_b\circ d^2\circ f_b^{-1}=0$ and hence we have the following cochain complex
\begin{center}
$C^0_b(SG)\stackrel{d_b}\To C^1_b(SG)\stackrel{d_b}\To C^2_b(SG)\stackrel{d_b}\To\cdots\stackrel{d_b}\To C^i_b(SG)\stackrel{d_b}\To\cdots$.
\end{center}

\begin{defn}
We call the cohomology groups of the cochain complex above the \emph{balanced chromatic cohomology groups} of $SG$ and denote them by $H^{\bullet}_b(SG)$.
\end{defn}

Similar to Proposition \ref{proposition4.8}, the balanced chromatic cohomology groups $H^{\bullet}_b(SG)$ are independent of the choice of the order of edges. Now we give two examples. The first one, as a supplement of Example \ref{ex4.2}, calculates the balanced chromatic cohomology groups of the signed graph $SP_3$. The second one is devoted to calculate the balanced chromatic cohomology groups of $SP_2$, see Figure \ref{fig:5}.

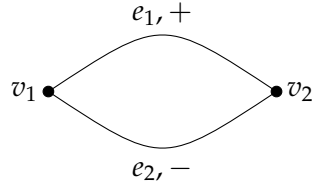
\begin{figure}[h]
\centering
\begin{tikzpicture}
\draw[fill] (-1.5,0) circle (.07) node [left] {$v_1$};
\draw[fill] (1.5,0) circle (.07) node [right] {$v_2$};
\draw (-1.5,0) ..controls (0,1)..node[above]{$e_1,+$} (1.5,0);
\draw (-1.5,0) ..controls (0,-1)..node[below]{$e_2,-$} (1.5,0);
\end{tikzpicture}
\caption{An unbalanced signed graph $SP_2$}
\label{fig:5}
\end{figure}

\begin{ex}
We use the same notions as that in Example \ref{ex4.2}. It's easy to find that for each $i\in\set{0,1,2}$, we have $C^i_b(SP_3)=C^i(SP_3)$, $B^i_b(SP_3)=B^i(SP_3)$ and $C^3_b(SP_3)=0$, $B^3_b(SP_3)=0$. On the other hand, $Z^0_b(SP_3)=Z^0(SP_3)$, $Z^1_b(SP_3)=Z^1(SP_3)$, $Z^2_b(SP_3)=C^2(SP_3)$, $Z^3_b(SP_3)=0$. Then we obtain the balanced chromatic cohomology groups as below
\begin{equation*}
\begin{aligned}
        H^0_b(SP_3)&=Z^0(SP_3)/B^0(SP_3)=H^0(SP_3)\cong\mathbb{Z}\{3\};\\
        H^1_b(SP_3)&=Z^1(SP_3)/B^1(SP_3)=H^1(SP_3)\cong\mathbb{Z}_2\{2\}\mathop\oplus\mathbb{Z}\{1\};\\
        H^2_b(SP_3)&=Z^2(SP_3)/B^2(SP_3)\cong\mathbb{Z}\{1\}\oplus\mathbb{Z};\\
        H^3_b(SP_3)&=0.
\end{aligned}
\end{equation*}
We can calculate the balanced chromatic polynomial of $SG$
\begin{center}
$P_{SP_3}(1+q)=\sum\limits_{i\geq0}(-1)^iq\text{dim}C^i_b(SP_3)=\sum\limits_{i\geq0}(-1)^iq\text{dim}H^i_b(SP_3)=q^3-q+q+1=q^3+1,$
\end{center}
where $q$ is odd. This result coincides with the result obtained in Example \ref{ex3.2}.
\end{ex}

\begin{ex}\label{ex:4.4}
Let $SP_2$ be an unbalanced signed graph on $P_2$, where $E(SP_2)=\set{e_1,e_2}$ and $\sigma(e_1)=+1$, $\sigma(e_2)=-1$, as Figure \ref{fig:5} shows. For each enhanced state $S=(s, c)$ of $SP_2$, we denote $s$ and $c$ similarly to Example \ref{ex4.2}, then we can calculate the cochain groups and the cohomology groups as follows.
\begin{equation*}
\begin{aligned}
    		C^0_b(SP_2)&=\text{span}\langle(00,xx),(00,1x),(00,x1),(00,11)\rangle;\\
    		C^1_b(SP_2)&=\text{span}\langle(01,xx),(10,xx),(01,11),(10,11)\rangle;\\
    		C^2_b(SP_2)&=0;\\
    		B^0_b(SP_2)&=0;\\
    		B^1_b(SP_2)&=\text{span}\langle(10,xx)+(01,xx),(10,11)+(01,11)\rangle;\\
    		B^2_b(SP_2)&=0;\\
    		Z^0_b(SP_2)&=\text{span}\langle(00,xx),(00,1x)-(00,x1)\rangle;\\
    		Z^1_b(SP_2)&=\text{span}\langle(01,xx),(10,xx),(01,11),(10,11)\rangle;\\
    		Z^2_b(SP_2)&=0;\\
    		H^0_b(SP_2)&=Z^0_b(SP_2)/B^0_b(SP_2)\cong\mathbb{Z}\set{2}\oplus\mathbb{Z}\set{1};\\
    		H^1_b(SP_2)&=Z^1_b(SP_2)/B^1_b(SP_2)\cong\mathbb{Z}\set{1}\oplus\mathbb{Z};\\
    		H^2_b(SP_2)&=Z^2_b(SP_2)/B^2_b(SP_2)\cong0.\\
\end{aligned}
\end{equation*}

According to Proposition \ref{proposition4.11}, the balanced chromatic polynomial can be calculated as
\begin{center}
$P_{SG}(1+q)=\sum\limits_{i\geq0}(-1)^iq\text{dim}C^i_b(SG)=\sum\limits_{i\geq0}(-1)^iq\text{dim}H^i_b(SG)=q^2+q-q-1=q^2-1$,
\end{center}
where $q$ is an odd integer. And we can check this result intuitively: suppose there are $1+q$ colors, if we assign one color to $v_1$, to which there are $1+q$ approaches. There are $q-1$ colors left to be chosen to color $v_2$. According to the multiplication principle, $P_{SP_2}(1+q)=(q+1)(q-1)=q^2-1$, if $q$ is odd. 
\end{ex}

We remark that it was proved that for unsigned graphs, the knight move conjecture holds for the chromatic cohomology with rational coefficients \cite{CCR}. More precisely, they proved that for a unsigned graph $G$, the nontrivial cohomology groups come in isomorphic pairs: $H^{i, |V(G)|-i}(G;\mathbb{Q})\cong H^{i+1, |V(G)|-i-2}(G;\mathbb{Q})$. According to the examples above, we find that this is not true for signed chromatic cohomology.

\subsubsection{Another long exact sequence}
As an analogy to Corollary \ref{corollary4.10}, the deletion-contraction relation for balanced chromatic polynomial also corresponds to a long exact sequence. This can be derived from the following theorem.

\begin{thm}\label{theorem4.16}
Let $SG$ be a signed graph and $e$ a positive edge, we have the following short exact sequence
\begin{center}
$0\to C_b^{\bullet-1}(SG/e)\to C_b^{\bullet}(SG)\to C_b^{\bullet}(SG-e)\to0$.
\end{center}
\end{thm}
\begin{proof}
The proof is a mimic of that of Theorem \ref{theorem4.9}. For cochain groups, we still have 
\begin{center}
$C_b^i(SG)=\bigoplus\limits_{e\notin s\subseteq E(SG), |s|=i}M_s^b(SG)\oplus\bigoplus\limits_{e\in s\subseteq E(SG), |s|=i}M_s^b(SG)=C_b^i(SG-e)\oplus C_b^{i-1}(SG/e),$
\end{center}
where $e$ is set as the first edge. The morphism $\widetilde{m_b}: C^{\bullet}_b(SG-e)\to C^{\bullet}_b(SG/e)$ can be similarly defined as
\begin{equation*}
\widetilde{m_b}=
\begin{cases}
f_b\circ id\circ f_b^{-1}, &\text{if \emph{e} connects one component to itself} \\
f_b\circ m\circ f_b^{-1}, &\text{if \emph{e} connects two different components}\\
\end{cases}
\end{equation*}
where $m$ denotes the multiplication on $M$ and $f_b$ is the map introduced in Definition \ref{definition4.13}. Let us use $d_b^1$ and $d_b^2$ to denote the differential maps on $C^{\bullet}(SG-e)$ and $C^{\bullet}(SG/e)$ respectively. It is not difficult to check that the differential 
$d_b'=\left(
\begin{array}{ccc}
d_b^1 &   0   \\
\widetilde{m_b} & -d_b^2   \\
\end{array}
\right)$
defined on $C_b^{\bullet}(SG-e)\bigoplus C_b^{\bullet-1}(SG/e)$ coincides with the differential $d_b$ defined on $C_b^{\bullet}(SG)$. In other words, the complex $C_b^{\bullet}(SG)$ is the mapping cone of $\widetilde{m_b}: C_b^{\bullet}(SG-e)\to C_b^{\bullet}(SG/e)$. The result follows.
\end{proof}

\begin{cor}\label{corollary4.17}
Let $SG$ be a signed graph and $e$ a positive edge, we have the following long exact sequence
\begin{center}
$\cdots\to H_b^{i-1}(SG/e)\to H_b^i(SG)\to H_b^i(SG-e)\to H_b^i(SG/e)\to\cdots$.
\end{center}
\end{cor}

\section{Some properties}\label{section5}
\subsection{Relation to the chromatic cohomology groups of unsigned graphs}
In Section \ref{section4}, we introduced two cochain complexes for signed graphs based on the categorification of the chromatic polynomial of unsigned graphs, which was proposed by Laure Helme-Guizon and Yongwu Rong in \cite{LY}. So it's natural to consider the relation among the cohomology groups of these three cochain complexes. Obviously, if all the edges of a signed graph are positive, then all these cohomology groups coincide. This conclusion can be enhanced a little bit as follows.

For a signed graph $SG$, let us use $n_b(SG)$ to denote the length of the shortest unbalanced circuits. In other words, any $[SG: s]$ is balanced provided that $|s|\leq n_b(SG)-1$. 

\begin{prop}\label{proposition5.1}
For any $i\leq n_b-2$, we have $H^i(SG)=H^i(G)=H^i_b(SG)$.
\end{prop}
\begin{proof}
For any $i\geq n_b$, by replacing all the cochain groups $C^{i}(G), C^{i}(SG)$ and $C^{i}_b(SG)$ with 0, one obtains three new cochain complexes. Since there is no negative circuit now, these three cochain complexes are exactly the same, hence the cohomology groups are isomorphic mutually.
\end{proof}

\subsection{Disjoint union of two signed graphs}
Let $SG_1$ and $SG_2$ be two signed graphs, we denote their disjoint union by $SG_1\sqcup SG_2$. On the chain complex level, we have $C^{\bullet}(SG_1\sqcup SG_2)=C^{\bullet}(SG_1)\otimes C^{\bullet}(SG_2)$ and $C^{\bullet}_b(SG_1\sqcup SG_2)=C^{\bullet}_b(SG_1)\otimes C^{\bullet}_b(SG_2)$. As a corollary of the K\"{u}nneth theorem, the cohomology groups of $SG_1$, $SG_2$ and their disjoint union satisfy the following relations.
\begin{prop}\label{proposition5.2}
For each $i\in\mathbb{N}$, we have
\begin{center} 
$H^i(SG_1\sqcup SG_2)\cong[\mathop\oplus\limits_{p+q=i}H^p(SG_1)\otimes H^q(SG_2)]\oplus[\mathop{\oplus}\limits_{p+q=i+1}H^p(SG_1)\ast H^q(SG_2)],$
$H^i_b(SG_1\sqcup SG_2)\cong[\mathop\oplus\limits_{p+q=i}H^p_b(SG_1)\otimes H^q_b(SG_2)]\oplus[\mathop{\oplus}\limits_{p+q=i+1}H^p_b(SG_1)\ast H^q_b(SG_2)],$
\end{center}
where $\ast$ is the torsion product of abelian groups.
\end{prop}

In particular, when $SG_2$ is a trivial graph, i.e. the graph with exactly one vertex, it is easy to find that $H^i(SG_2)=H^i_b(SG_2)=\mathbb{Z}\oplus\mathbb{Z}x$. It follows that
\begin{center}
$H^i(SG_1\sqcup SG_2)\cong H^i(SG_1)\otimes (\mathbb{Z}\oplus\mathbb{Z}x),$
$H^i_b(SG_1\sqcup SG_2)\cong H^i_b(SG_1)\otimes (\mathbb{Z}\oplus\mathbb{Z}x).$
\end{center}

\subsection{Vertex switching operation}
Recall that two signed graphs are equivalent if they are related by several vertex switchings. As we mentioned before, equivalent signed graphs have the same signed chromatic polynomial. The following result tells us that vertex switching not only preserves the signed chromatic polynomial, but also the signed chromatic cohomology groups.
\begin{prop}\label{proposition5.3}
If two signed graphs $SG_1, SG_2$ are equivalent, then $H^i(SG_1)\cong H^i(SG_2)$ and $H^i_b(SG_1)\cong H^i_b(SG_2)$.
\end{prop}
\begin{proof}
It suffices to consider the case that $SG_2$ is obtained from $SG_1$ by applying vertex switching on a vertex $v\in V(SG_1)$. Notice that for any $s\subseteq E(SG_1)$, a component in $[SG_1: s]$ is balanced if and only if the corresponding component in $[SG_2: s]$ is also balanced. This induces a cochain map from $C^{\bullet}(SG_1)$ to $C^{\bullet}(SG_2)$ and another cochain map from $C^{\bullet}_b(SG_1)$ to $C^{\bullet}_b(SG_2)$, both of which induce isomorphisms between the cohomology groups.
\end{proof}

\subsection{Contracting a pendant edge}
\begin{defn}
Let $SG$ be a signed graph, suppose $v\in V(SG)$ is a vertex of degree one. We call the edge incident with $v$ a \emph{pendant edge} of $SG$.
\end{defn}

For a given signed graph $SG$ and a positive pendant edge $e$ (one can apply vertex switching on $v$ if necessary), Proposition \ref{proposition3.1} tells us that
\begin{center}
$P_{SG}(\lambda)=P_{SG-e}(\lambda)-P_{SG/e}(\lambda)=(\lambda-1)P_{SG/e}(\lambda)$.
\end{center}
The following proposition can be seen as a categorification of this.

\begin{prop}\label{proposition5.5}
Let $e$ be a pendant edge in a signed graph $SG$. For each $i$, we have $H^i(SG)\cong H^i(SG/e)\{1\}$ and $H^i_b(SG)\cong H^i_b(SG/e)\{1\}$.
\end{prop}
\begin{proof}
We only prove $H^i(SG)\cong H^i(SG/e)\{1\}$, the balanced version can be proved analogously. By switching $v$ if necessary, we assume the sign of $e$ is positive and $e$ is the first edge. The key observation is, for any $s\subset E(SG)$ the pendant edge $e$ has no effect on the balanced components of $[SG: s]$. The main idea of the proof is similar to the unsigned case \cite{LY}. We sketch the outline here.

According to Corollary \ref{corollary4.10}, we have the following long exact sequence
\begin{center}
$\cdots\to H^{i-1}(SG/e)\to H^i(SG)\to H^i(SG-e)\to H^i(SG/e)\to\cdots$.
\end{center}
On the other hand, since $SG-e=SG/e\cup\{v\}$, Proposition \ref{proposition5.2} tells us that 
\begin{center}
$H^i(SG-e)\cong H^i(SG/e)\otimes(\mathbb{Z}\oplus\mathbb{Z}x)\cong H^i(SG/e)\oplus H^i(SG/e)\{1\}$. 
\end{center}
By identifying $H^i(SG-e)$ with $H^i(SG/e)\oplus H^i(SG/e)\{1\}$, it suffices to show that the map $\gamma^*: H^i(SG/e)\oplus H^i(SG/e)\{1\}\to H^i(SG/e)$ sends $(x, 0)$ to $x$. 

In fact, for any $x=[\sum\limits_ia_iS_i]\in H^i(SG/e)$, where $S_i=(s_i, c_i)$, we extend each $S_i$ to be an enhanced state in $C^i(SG-e)$ by adding an isolated vertex $v$ with color 1. Then the map $\gamma^*$ sends each $[(s_i, c_i)]$ to $[(s_i\cup\{e\})/e, (c_i)_e]$. Notice that adding a positive pendant edge preserves the balance of each component. On the other hand, since $v$ is colored by 1 and multiplication by 1 is just the identity map, it follows that $\gamma^*((x, 0))=x$. 

Therefore $\gamma^*$ is surjective and hence the long exact sequence splits into infinitely many short exact sequences
\begin{center}
$0\to H^i(SG)\to H^i(SG/e)\oplus H^i(SG/e)\{1\}\stackrel{\gamma^*}{\rightarrow}H^i(SG/e)\to0$.
\end{center}
It follows from Lemma 3.10 in \cite{LY} that $H^i(SG)\cong H^i(SG/e)\{1\}$.
\end{proof}

\subsection{Loops and parallel edges}
We first discuss the effect of positive/negative loops on the signed chromatic cohomology. Similar to the fact that signed graphs with positive loops have zero signed chromatic polynomial, positive loops also kill the signed chromatic cohomology.
\begin{prop}\label{proposition5.6}
If a signed graph $SG$ has a positive loop $e$, then $H^i(SG)=H^i_b(SG)=0$.
\end{prop}
\begin{proof}
The assumption $e$ is a positive loop implies that $SG/e=SG-e$. By investigating the map $\gamma^*$ in the following long exact sequence
\begin{center}
$\cdots\to H^{i-1}(SG/e)\to H^i(SG)\to H^i(SG-e)\stackrel{\gamma^*}{\rightarrow}H^i(SG/e)\to H^{i+1}(SG)\to\cdots$,
\end{center}
it is not difficult to find that $\gamma^*$ is an isomorphism. We conclude that $H^i(SG)=0$. The unbalanced case can be proved similarly.
\end{proof}

\begin{prop}\label{proposition5.7}
If a signed graph $SG$ has a negative loop $e$, then $H^i_b(SG)=H^i_b(SG-e)$.
\end{prop}
\begin{proof}
Since $e$ is negative, Corollary \ref{corollary4.17} does not work here. However, if we go back to the cochain complex, we have the following decomposition
\begin{center}
$C_b^i(SG)=\bigoplus\limits_{e\notin s\subseteq E(SG), |s|=i}M_s^b(SG)\oplus\bigoplus\limits_{e\in s\subseteq E(SG), |s|=i}M_s^b(SG)$.
\end{center}
If a subset $s\subseteq E(SG)$ includes the negative loop $e$, then $[SG: s]$ is unbalanced, therefore the associated $M_s^b(SG)=0$. It follows that the second summand above vanishes and the result follows immediately.
\end{proof}

Now we turn to discuss the effect of parallel edges on signed chromatic cohomology. Recall that two edges join the same pair of vertices, then these two edges are called \emph{parallel edges}. Here we allow the two endpoints coincide with each other.

\begin{prop}\label{proposition5.8}
Let $SG$ be a signed graph, and $e, e'\in E(SG)$ are a pair of parallel edges with the same sign. Then we have $H^i(SG)=H^i(SG-e')$ and $H^i_b(SG)=H^i_b(SG-e')$. In other words, the signed chromatic cohomology groups are unchanged if one replaces the parallel edges with the same sign by a single one.
\end{prop}
\begin{proof}
If both $e$ and $e'$ are positive, we divide our discussion into two cases.
\begin{itemize}
\item The two endpoints of $e$ and $e'$ coincide, i.e. both $e$ and $e'$ are positive loops. In this case, the result follows from Proposition \ref{proposition5.6}.
\item The two endpoints of $e$ and $e'$ are distinct. Then $e$ becomes a positive loop in $SG/e'$, hence we obtain $H^i(SG/e')=H^i_b(SG/e')=0$. It follows from the two long exact sequences that $H^i(SG)=H^i(SG-e')$ and $H^i_b(SG)=H^i_b(SG-e')$.
\end{itemize}

If both $e$ and $e'$ are negative, there are also two situations.
\begin{itemize}
\item Both $e$ and $e'$ are not loops. By switching one of the two endpoints of $e$ and $e'$ we obtain two positive parallel edges, which has been discussed above.
\item Both $e$ and $e'$ are loops. The balanced case $H^i_b(SG)=H^i_b(SG-e')$ follows directly from Proposition \ref{proposition5.7}. The rest of the proof is devoted to show that $H^i(SG)=H^i(SG-e')$. Suppose both $e$ and $e'$ connects $v\in V(SG)$ to itself. We define a new signed graph $SG'$ by splitting $v$ into two vertices, say $v_1, v_2$, and adding a new positive edge $e_0$ which connects $v_1$ and $v_2$. All edges incident to $v$ in $SG$ are now incident to $v_1$ in $SG'$. See Figure \ref{figure5}.
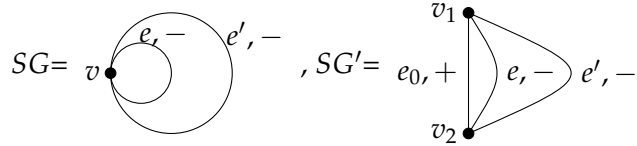
\begin{figure}[h]
	\centering
		$SG$=
		\begin{tikzpicture}[baseline=0]
			\draw[fill] (0,0) circle (.07) node [left] {$v$};
			\draw (0.4,0) circle (.4);
			\draw (0.8,0) circle (.8);
			\node at (0.7,0.5) {$e,-$};
			\node at (1.9,0.5) {$e',-$};
		\end{tikzpicture}
		, $SG'$=
		\begin{tikzpicture}[baseline=0]
			\draw[fill] (0,0.8) circle (.07) node [left] {$v_1$};
			\draw[fill] (0,-0.8) circle (.07) node [left] {$v_2$};
			\draw (0,0.8)-- node [left] {$e_0,+$} (0,-0.8);
			\draw (0,0.8)..controls (0.5,0)..node [right] {$e,-$} (0,-0.8);
			\draw (0,0.8)..controls (1.8,0)..node [right] {$e',-$} (0,-0.8);
		\end{tikzpicture}
\caption{Adding a new vertex}
\label{figure5}
\end{figure}

Now we have $SG=SG'/e_0, SG-e'=(SG'-e')/e_0$ and $e, e'$ are negative edges connecting $v_1$ and $v_2$ in $SG'$ and $SG'-e_0$. According to our previous discussion, if we switch $v_2$, delete $e'$ and then switch $v_2$ again, these operations induce isomorphisms $H^i(SG')\cong H^i(SG'-e')$ and $H^i(SG'-e_0)\cong H^i(SG'-e_0-e')$. The homomorphism $\delta^*: H^i(SG'/e_0)\to H^i((SG'-e')/e_0)$ in the commutative diagram below can be defined as follows. Given an enhanced state $S=(s, c)$ of $SG'/e_0$, if $e'\notin s$, then we define $\delta(S)=S$, which is also an enhanced state of $(SG'-e')/e$. Otherwise, we define $\delta(S)=0$. This map makes the diagram below commutative and induces a homomorphism $\delta^*: H^i(SG'/e_0)\to H^i((SG'-e')/e_0)$.
\begin{equation*}
\begin{aligned}
		\xymatrix{
		H^i(SG')\ar[d]^\cong \ar[r] & H^i(SG'-e_0)\ar[d]^\cong \ar[r] & H^i(SG'/e_0)\ar[d]^{\delta^*}\ar[r] & H^{i+1}(SG')\ar[d]^\cong \ar[r] & H^{i+1}(SG'-e_0)\ar[d]^\cong\\
		H^i(SG'-e')\ar[r] & H^i(SG'-e'-e_0)\ar[r] & H^i((SG'-e')/e_0)\ar[r] & H^{i+1}(SG'-e')\ar[r] & H^{i+1}(SG'-e'-e_0)}
\end{aligned}
\end{equation*}
We conclude that
\begin{center}
$H^i(SG)\cong H^i(SG'/e_0)\cong H^i((SG'-e')/e_0)\cong H^i(SG-e')$,
\end{center}
the fact that $\delta^*$ is an isomorphism is derived from the five lemma. The proof is finished.
\end{itemize}
\end{proof}

\subsection{Trees and polygon graphs}
As an application of the properties discussed above, we describe the signed chromatic cohomology groups for some classes of signed graphs.
\begin{ex}\label{ex:SNmn}
Let $SN_m^n$ be the signed graph with $m$ vertices, on $n$ of which there is a negative loop. When $m=n=1$, we can calculate the signed chromatic cohomology groups for $SN_1^1$ as follows.
\begin{equation*}
\begin{aligned}
			C^0(SN_1^1)&=\text{span}\langle(0,1),(0,x)\rangle &C^1(SN_1^1)&=\text{span}\langle(1,1)\rangle\\
			%C^1(SN_1^1)&=\mathop{span}\langle(1,1)\rangle\\
			B^0(SN_1^1)&=0 &B^1(SN_1^1)&=\text{span}\langle(1,1)\rangle\\
			%B^1(SN_1^1)&=\mathop{span}\langle(1,1)\rangle\\
			Z^0(SN_1^1)&=\text{span}\langle(0,x)\rangle &Z^1(SN_1^1)&=\text{span}\langle(1,1)\rangle\\
			%Z^1(SN_1^1)&=\mathop{span}\langle(1,1)\rangle\\
			H^0(SN_1^1)&=Z^0(SN_1^1)/B^0(SN_1^1)\cong\mathbb{Z}x &H^1(SN_1^1)&=Z^1(SN_1^1)/B^1(SN_1^1)\cong0\\
			%H^1(SN_1^1)&=Z^1(SN_1^1)/B^1(SN_1^1)\cong0\\
			C^0_b(SN_1^1)&=\text{span}\langle(0,1),(0,x)\rangle &C^1_b(SN_1^1)&=0\\
			%C^1_b(SN_1^1)&=0\\
			B^0_b(SN_1^1)&=0 &B^1_b(SN_1^1)&=0\\
			%B^1_b(SN_1^1)&=0\\
			Z^0_b(SN_1^1)&=\text{span}\langle(0,1),(0,x)\rangle &Z^1_b(SN_1^1)&=0\\
			%Z^1_b(SN_1^1)&=0\\
			H^0_b(SN_1^1)&=Z^0_b(SN_1^1)/B^0_b(SN_1^1)\cong\mathbb{Z}\oplus\mathbb{Z}x &H^1_b(SN_1^1)&=Z^1_b(SN_1^1)/B^1_b(SN_1^1)\cong0\\
			%H^1_b(SN_1^1)&=Z^1_b(SN_1^1)/B^1_b(SN_1^1)\cong0\\
\end{aligned}
\end{equation*}
Then Proposition \ref{proposition5.2} implies that
	\[
	\begin{aligned}
		H^i(SN_m^n)&\cong
	\begin{cases}
		(\mathbb{Z}\oplus\mathbb{Z}\{1\})^{\otimes(m-n)}\otimes(\mathbb{Z}\{1\})^{\otimes n}, & i=0;\\
		0, & i\geq 1.\\
	\end{cases}\\
	H^i_b(SN_m^n)&\cong
	\begin{cases}
		(\mathbb{Z}\oplus\mathbb{Z}\{1\})^{\otimes m}, & i=0;\\
		0, & i\geq 1.\\
	\end{cases}\\
	\end{aligned}
	\]
We remark that the fact $H^i_b(SN_m^n)\cong H^i_b(SN_m^0)$ also can be deduced from Proposition \ref{proposition5.7}.
\end{ex}

\begin{ex}
Let $ST_n=(T_n,\sigma)$, where $T_n$ is a tree with $n$ edges. By repeatedly using Proposition \ref{proposition5.5} one obtains
\begin{center}
$H^i(ST_n)\cong H^i_b(ST_n)\cong 
\begin{cases}
\mathbb{Z}\{n\}\oplus\mathbb{Z}\{n+1\}, & i=0;\\
0, & i\geq 1.\\
\end{cases}$
\end{center}
\end{ex}

\begin{ex}
Let $SP_n=(P_n,\sigma)$ be an unbalanced polygon graph with $n$ edges. When $n=1$, the case $SP_1=SN_1^1$ has been discussed in Example \ref{ex:SNmn}. For this reason, next let us assume $n\geq 2$. 

We remark that although unbalanced polygon graphs on the unsigned polygon graph $P_n$ are not unique, they are all equivalent. In order to see this, first notice that two unbalanced polygon graphs with only one negative edge are equivalent. If an unbalanced polygon graph has more than three negative edges, choose two of them such that we can find a positive path connecting them. By applying vertex switching on the endpoints of this positive path we obtain a new unbalanced polygon graph with two negative edges less.
	
We label the vertices of $SP_n$ by $v_1,v_2,\cdots, v_n$ monotonically so that each $v_i$ is adjacent to $v_{i\pm1}$ $(1\leq i\leq n)$, where $v_0=v_n$ and $v_{n+1}=v_1$. Let $e$ be a positive edge connecting $v_1$ and $v_n$, the $SP_n/e=SP_{n-1}$ and $SP_n-e$ is a tree. Then we have the following long exact sequence
\begin{center}
$\cdots\to H^{i-1}(SP_n-e)\to H^{i-1}(SP_n/e)\to H^{i}(SP_n)\to H^i(SP_n-e)\to\cdots$.
\end{center}
As $SP_n-e$ is a tree, we have $H^i(SP_n-e)=0$ for all $i\geq 1$. Thus for any $i\geq2$, we have $H^i(SP_n)\cong H^{i-1}(SP_n/e)\cong H^{i-1}(SP_{n-1})$, and it follows that
\begin{equation*}
H^i(SP_n)\cong
\begin{cases}
H^1(SP_{n-i+1}), & \mbox{if } 2\leq i\leq n;\\
0, & \mbox{if } i>n.
\end{cases}
\end{equation*}
Then for each $n\geq 2$, $H^n(SP_n)\cong H^1(SP_1)\cong 0$, $H^{n-1}(SP_n)\cong H^2(SP_3)\cong\mathbb{Z}\{1\}$, which has been calculated in Example \ref{ex4.2}.

On the other hand, it has been calculated in \cite{LY} that
\[\text{For $i=0$, }H^{0}\left(P_{n}\right) \cong\left\{\begin{array}{ll}
\mathbb{Z}\{n\} \oplus \mathbb{Z}\{n-1\} & \text { if } n \text { is even and } n \geq 2; \\
\mathbb{Z}\{n\} & \text { if } n \text { is odd and } n \geq 2; \\
0 & \text { if } n=1.
\end{array}\right.\]
\begin{equation*}
\text{For $i>0$, }H^{i}\left(P_{n}\right) \cong\left\{\begin{array}{ll}
\mathbb{Z}_{2}\{n-i\} \oplus \mathbb{Z}\{n-i-1\} & \text { if } n-i \geq 2 \text { and $n$ is even;} \\
\mathbb{Z}\{n-i\} & \text { if } n-i \geq 2 \text { and $n$ is odd;} \\
0 & \text { if } n-i \leq 1.
\end{array}\right.
\end{equation*}
And Using Proposition \ref{proposition5.1}, we have $H^i(SP_n)=H^i(P_n)$ for all $i\leq n-2$, then we obtain all the cohomology groups.
\[\text{For $i=0$, }H^{0}\left(SP_{n}\right) \cong\left\{\begin{array}{ll}
\mathbb{Z}\{n\} \oplus \mathbb{Z}\{n-1\} & \text { if } n \text { is even and } n \geq 2;\\
\mathbb{Z}\{n\} & \text { if } n \text { is odd and } n \geq 2;\\
\mathbb{Z}\{1\} & \text { if } n=1.
\end{array}\right.\]
\begin{equation*}
\text{For $i>0$, }H^{i}\left(SP_{n}\right) \cong\left\{\begin{array}{ll}
\mathbb{Z}_{2}\{n-i\} \oplus \mathbb{Z}\{n-i-1\} & \text { if } i\leq n-2 \text { and $n$ is even;} \\
\mathbb{Z}\{n-i\} & \text { if } i\leq n-2 \text { and $n$ is odd;} \\
\mathbb{Z}\{1\} & \text{ if } i=n-1;\\
0 & \text { if } i\geq n.
\end{array}\right.
\end{equation*}
We work out the balanced cohomology groups parallelly.
\[\text{For $i=0$, }H^{0}_b\left(SP_{n}\right) \cong\left\{\begin{array}{ll}
\mathbb{Z}\{n\} \oplus \mathbb{Z}\{n-1\} & \text { if } n \text { is even and } n \geq 2; \\
\mathbb{Z}\{n\} & \text { if } n \text { is odd and } n \geq 2; \\
\mathbb{Z}\{1\}\oplus\mathbb{Z} & \text { if } n=1.
\end{array}\right.\]
\begin{equation*}
\text{For $i>0$, }H^{i}_b\left(P_{n}\right) \cong\left\{\begin{array}{ll}
\mathbb{Z}_{2}\{n-i\} \oplus \mathbb{Z}\{n-i-1\} & \text { if } i\leq n-2 \text { and $n$ is even;} \\
\mathbb{Z}\{n-i\} & \text { if } i\leq n-2 \text { and $n$ is odd;} \\
\mathbb{Z}\{1\}\oplus\mathbb{Z} & \text{ if } i=n-1;\\
0 & \text { if } i\geq n.
\end{array}\right.
\end{equation*}
\end{ex}

\section*{Acknowledgements}
ZY Cheng is supported by NSFC 11771042 and NSFC 12071034. ZY Lei, YT Wang and YG Zhang are supported by an undergraduate research project of Beijing Normal University.

% ----------------------------------------------------------------

\end{document}